\documentclass[11pt]{amsart}
\usepackage{a4, amsmath,  amsthm,  amsfonts, mathrsfs, latexsym, amssymb, 
pstricks, pst-grad, graphicx, tabularx, setspace, longtable}
\usepackage{enumerate}
\usepackage{hyperref}
\usepackage{color}
\usepackage{psfrag}
\usepackage{amssymb}
\scshape
\date{}
\usepackage{tikz}
\usetikzlibrary{arrows}
\usepackage{subcaption}
\usetikzlibrary{automata}
\newtheorem{thm}{\bf Theorem}[section]

\newtheorem{defn}[thm]{\bf Definition}

\newtheorem{rem}[thm]{\bf Remark}

\begin{document}
\title[Character degree graph of solvable groups with odd degree]{CHARACTER DEGREE GRAPH OF SOLVABLE GROUPS WITH ODD DEGREE}
 
\author [G. Sivanesan]{G. Sivanesan}
\address{Department of Mathematics, Government College of Engineering, Salem 636011, Tamil Nadu, India, ORCID: 0000-0001-7153-960X.}
\email{sivanesan@gcesalem.edu.in}

\author [C. Selvaraj]{C. Selvaraj}
\address{Department of Mathematics, Periyar University, Salem 636011, Tamil Nadu, India,  ORCID: 0000-0002-4050-3177.}
\email{selvavlr@yahoo.com}

\keywords{character graph, Eulerian graph, non-regular graph, finite solvable group}
\subjclass[2020]{05C45, 20F16, 20C15}
 
\begin{abstract}Let $G$  be a finite group,  let $Irr(G)$ be the set of all complex irreducible characters of $G$ and let $cd(G)$ be the set of all degrees of characters in $Irr(G).$  Let $\rho(G)$ be the set of primes that divide degrees in $cd(G).$ The character degree graph $\Delta(G)$ of $G$ is the simple undirected graph with vertex set $\rho(G)$ and in which two distinct vertices $p$ and $q$ are adjacent  if there exists a  character degree $r \in cd(G)$ such that $r$ is divisible by the product $pq.$ In this paper, we obtain a necessary condition for the character degree graph $\Delta(G)$ with all of its vertices are odd degree  of a finite solvable group $G$ .
\end{abstract}
\date{}
\maketitle
 
\section{Introduction}  Throughout this paper, $G$ will be a finite solvable group with identity $1.$ We denote the set of complex irreducible characters of $G$ by $Irr(G).$ Here $cd(G) =\{\chi(1) \mid \chi\in Irr(G)\}$ is the set of all distinct degrees of irreducible characters in $Irr(G).$ Let $\rho(G)$ be the set of all primes that divide degrees in $cd(G).$  There is a vast literature  devoted to study of ways through which one can associate a graph with a group and that literature can be used  for investigating the algebraic structure of groups  using graph theoretical properties of  associated graphs. One of these graphs is the character degree graph $\Delta(G)$ of $G$.   In fact, $\Delta(G)$ is an undirected simple graph with vertex set $\rho(G)$ in which $p,q\in\rho(G)$ are joined by an edge if there exists a character degree $\chi(1)\in cd(G)$ which is divisible by $pq.$  This graph was first defined in  \cite{Manz2} and studied by many authors (see \cite{Lewis4},\cite {PP1} and\cite{CBS}). When $G$ is a solvable group, some interesting results on the character graph of $G$ have been obtained. Actually, Manz  \cite{Manz1} proved that $\Delta(G)$ has at most two connected components. Also, Manz et al.\cite{Manz3} have proved that diameter of $\Delta(G)$ is at most $3$. Mahdi Ebrahimi et al.  \cite{EI} proved that the character graph $\Delta(G)$ of a solvable group $G$ is Hamiltonian if and only if $\Delta(G)$ is a block with at least $3$ vertices.   \cite{SST} We obtain a necessary condition for the character degree graph $\Delta(G)$  of a finite solvable group $G$ to be Eulerian.   Motivated by these studies on character degree graphs of solvable groups, we find a  necessary condition for the character degree graph $\Delta(G)$ of a solvable group $G$ with odd degree vertices only.
\section{Preliminaries}
In this section, we present some preliminary results which are used in the paper. All graphs are assumed to be simple, undirected and finite. Let $\Gamma$ be a graph with vertex set $V(\Gamma)$ and edge set $E(\Gamma)$. If $\Gamma$ is connected, then the distance $d(u,v)$ between two distinct vertices $u,v\in V(\Gamma)$ is the length of the shortest path between them. The supremum of all distance between possible pairs of distinct vertices is known as the diameter of the graph. A complete graph with $n$ vertices in which any of the two distinct vertices are adjacent is denoted by $ K_n.$  A cycle with $n$ vertices is denoted by $C_n.$ The vertex connectivity $k(\Gamma)$ of $\Gamma$ is defined to be the minimum number of vertices whose removal from $\Gamma$ results in a disconnected subgraph of $\Gamma.$ A cut vertex $v$ of a graph $\Gamma$ is a vertex  such that the number of connected components of $\Gamma -v$ is more than the number of connected components of $\Gamma$. Similar definition is applicable in the case of cut edge of graph $\Gamma.$  A maximal connected subgraph without a cut vertex is called a block. By their maximality, different blocks of $\Gamma$ overlap in at most one vertex, which is then a cut vertex. Thus, every edge of $\Gamma$ lies in a unique block and $\Gamma$ is the union of its blocks. The degree of a vertex $v$ in $\Gamma$ is the number of edges incident with $v$ and the same is denoted by $ d(v)$ or deg$v$. A graph $\Gamma$ is called $k$-regular, if the degree of each vertex is $k$. 
A graph $\Gamma$ is said to be Eulerian if it contains a cycle containing all vertices of $\Gamma.$
We will take into account the following well known facts concerning character degree graphs and they are needed in the next section.

\begin{rem}\normalfont\label{rem2.1} For results regarding $\Delta(G),$ we start with P\'{a}lfy's three prime theorem on the character degree graph of solvable groups. P\'{a}lfy theorem~\cite[Theorem, p. 62]{PP} states that given a solvable group $G$ and any three distinct vertices of $\Delta(G)$, there exists an edge incident with the other two vertices. On applying P\'{a}lfy's theorem, $\Delta(G)$  has at most two connected components. 
\end{rem}
\begin{rem}\normalfont\label{rem2.2} Let $G$ be a solvable group. Then $diam(\Delta(G))\leq 3$    \cite  [Theorem 3.2]{Manz3}. Assume that the diameter $diam(\Delta(G))=3$ and let $r,s\in\rho(G)$ be two distinct vertices in $\Delta(G)$ such that distance $d(r,s)=3.$  When $\Delta(G)$ has exactly diameter $3$ and contains at least $5$ vertices,  Lewis~\cite[p. 5487]{Lewis2} proved that the vertex set $\rho(G)$ of $\Delta(G)$ can be partitioned into $\rho_1,\rho_2,\rho_3$ and $\rho_4.$ Actually $\rho_4$ is the set of all vertices of $\Delta(G)$ which are at distance $3$ from the vertex $r$ and so we have that $s\in \rho_4,$ $\rho_3$ is  the set of all vertices of $\Delta(G)$ which are distance $2$ from the vertex $r,$ $\rho_2$ is the set of all vertices that are adjacent to vertex $r$ and adjacent to some prime in $\rho_3$ and $\rho_1$ consists of $r,$ the set of all vertices which are adjacent to $r$ and not adjacent to any vertex in $\rho_3.$  This implies that, no vertex in $\rho_1$ is adjacent to any vertex in $\rho_3\cup\rho_4$ and no vertex in $\rho_4$ is adjacent to any vertex in $\rho_1\cup\rho_2,$ every vertex in $\rho_2$ is adjacent to some vertex in $\rho_3$ and vice versa, and $\rho_1\cup\rho_2$ and $\rho_3\cup\rho_4$ both induce complete subgraphs of $\Delta(G).$
\end{rem}
\begin{rem}\normalfont\label{rem2.3} 
Huppert    \cite[p. 25]{HU} listed all possible graphs $\Delta(G)$ for solvable groups $G$ with at most $4$ vertices.   In fact, every graph with $3$ or few vertices that satisfies P\'{a}lfy's condition occurs as $\Delta(G)$ for some solvable group $G.$ 
\end{rem}
\begin{thm}\cite[Lemma 2.7]{EI}\label{2.2} Let $G$ be a group with $\mid\rho(G)\mid\geq 3.$  If $\Delta(G)$ is not a block and the diameter of $\Delta(G)$ is at most $2$, then each block of $\Delta(G)$ is a complete graph.
\end{thm}

\begin{thm} \cite[Theorem 5]{Zha}\label{2.3}  The graph with four vertices in Figure 1 is not the character degree graph of a solvable group.
	
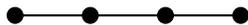
\begin{figure}[ht]
\centering
\begin{tikzpicture}
\draw[fill=black] (0,0) circle (3pt);
 \draw[fill=black] (1,0) circle (3pt);
 \draw[fill=black] (2,0) circle (3pt);
 \draw[fill=black] (3,0) circle (3pt);
\draw[thick] (0,0) --(1,0) --(2,0) --(3,0);
\end{tikzpicture}
 \caption{Graph with four vertices}
\end{figure}
\end{thm}

\begin{thm}\cite[Theorem 1.1]{Lewis5}\label{2.4}  Let $G$ be a solvable group. Then $\Delta(G)$ has at most one cut vertex.
\end{thm}	

\begin{thm}\cite[Theorem A]{MZ}\label{2.4.1} If $\Delta(G)$ is a non-complete and regular character degree graph of a finite solvable group $G$ with $n$ vertices, then $\Gamma(G)$ is $n-2$ is a regular graph.
\end{thm}

\begin{thm}\cite[Lemma 2.1]{Lewis5}\label{2.5}  Let $G$ be a solvable group and assume that $\Delta(G)$ has diameter $3$. Then $G$ is $1$-connected if and only if $\mid\rho_2\mid=1$ in the diameter $3$ partition of  $\rho(G)$. In this case, if $p$ is the unique prime in $\rho_2$, then $p$ is also the unique cut vertex for $\Delta(G)$. In particular, $\Delta(G)$ has at most one cut vertex.
\end{thm}

\begin{thm} \cite[Corollary B]{Lewis31}\label{2.6}  If $G$ is a solvable group with $ n= \mid \rho(G) \mid $ and $\Delta(G)$  contains two vertices of degree less than $ n-2$ that are not adjacent, then the Fitting height of $G$ is at least $3$.
\end{thm}

\begin{thm}\cite[Lemma 4.1]{EII}\label{2.7}
Let $G$ be a solvable group  where the diameter of  $\Delta(G)$ is $3$ with Lewis partition $\rho(G)=\rho_1\cup\rho_2\cup\rho_3\cup\rho_4.$ Then $\Delta(G)$ is a block if and only if $\mid\rho_2\mid,\mid\rho_3\mid\geq 2.$
\end{thm}

\begin{defn}\cite{Lewis3}\label{def2.8} Using direct product, one can find bigger groups from smaller groups. The same may be used to construct higher order character degree graphs. Given groups $A$ and $B,$ we have that $\rho(A\times B) =  \rho(A) \cup \rho(B).$ Define an edge between vertices $p$ and $q$ in $\rho(A\times B)$  if any of the following is satisfied:
\begin{itemize}
\item [\rm (i)] $p , q\in \rho (A)$ and there is an edge between $p$ and $q$ in $\Delta(A);$
\item [\rm (ii)] $p, q \in \rho (B)$ and there is an edge between $p$ and $q$ in $\Delta(B);$
\item [\rm (iii)] $p\in\rho (A)$ and $q\in\rho (B);$
\item [\rm (iv)]  $p \in\rho (B) $ and $q\in\rho (A).$ 
\end{itemize}
Now we get a higher order character degree graph and it is called direct product.
\end{defn}
\section{ Character Degree Graphs with odd degree}

\begin{thm}\label{3.1}
Let $G$ be a solvable group and $\Delta(G)$ is a complete graph with  $n \geq 2$ is even, then  $\Delta(G)$ is a graph with odd degree.
\end{thm}
\begin{proof}
Every complete graph is a character degree graph \cite[Lemma 2.1]{BL1}. By assumption, $n$ is even. So all the vertices of  $\Delta(G)$ are odd degree.\hfill $\square$ 
\end{proof}
\begin{thm}\label{3.2}
 If $\Delta(G)$ is a non-complete and regular character degree graph of a finite solvable group $G$ with $n$ vertices, then $\Delta(G)$ is not a character degree graph with odd degrees .
\end{thm}
\begin{proof}
By \cite[Theorem A]{MZ}, then  $\Delta(G)$ is $n-2$ regular graph. 
\newline  \textbf{Case 1.} $n$ is even.
\newline If $n$ is even, then the regular graph is even. So $\Delta(G)$ is not a character degree graph with odd degree
\newline  \textbf{Case 2.} $n$ is odd.
\newline The number of vertices of odd degree is even. So $n$ can not be odd.\hfill $\square$
\end{proof}

\begin{thm}\label{3.3}
Let $G$ be a solvable group  and $\Delta(G)$ is a graph with odd degree, then $\Delta(G)$ is a block .
\end{thm}
\begin{proof}
\newline  \textbf{Case 1.} For, assume that $\Delta(G)$ has diameter $3$. By assumption $\Delta(G)$ is not a block, and each block of $\Delta(G)$ contain even number of vertices. $\Delta(G)$ has two blocks say $B_1$ and $B_2$ and it has a cut vertex. This cut vertex is belongs to $\rho_2$. By  Theorem~\ref{2.5}, $\mid\rho_2\mid =1$. Since $\rho_1 \cup \rho_2$ and $\rho_3 \cup \rho_4$ induce complete subgraphs, each vertex  in  $\rho_3 \cup \rho_4$ is of odd degree.  Since $\mid\rho_2\mid = 1$, and every vertex in  $\rho_3$ is adjacent to some vertex in $\rho_2$ by Remark~\ref{rem2.2}. So each vertex  in $\rho_3$ is of even degree.  Hence  $\Delta(G)$ is not a graph  with odd degree vertices.
\newline  \textbf{Case 2.} Suppose $\Delta(G)$ has diameter at most $2.$  By assumption $\Delta(G)$ is not a block, and each block of $\Delta(G)$ contain even number of vertices. $\Delta(G)$ has two blocks say $B_1$ and $B_2$ and it has a cut vertex. By Theorem~\ref{2.2}, each block of $\Delta(G)$ is a complete graph. So, the cut vertex of $\Delta(G)$ is even degree. Therefore  $\Delta(G)$ is not a graph  with odd degree vertices. \hfill $\square$
\end{proof}

We  prove  that the character degree graph $\Delta(G)$ is a graph  with odd degree vertices only when the diameter of $\Delta(G)$ is $3.$
 
\begin{thm}\label{3.2} Let $G$ be a solvable group with diameter of $\Delta(G)$ is $3.$ Let the  Lewis' partition of $G$ be $\rho(G)=\rho_1\cup\rho_2\cup\rho_3\cup\rho_4.$ Then  $\Delta(G)$ is a graph with only odd degree vertices   if and only if the following conditions are hold:
\begin{itemize}
\item [\rm (i)] $\Delta(G)$ is a block;
\item [\rm (ii)] Both $\mid\rho_1 \cup \rho_2\mid$ and $\mid\rho_3 \cup \rho_4\mid$ are even;
\item [\rm (iii)] The subgraph of $\Delta(G)$ induced by $\rho_2 \cup \rho_3$ is Eulerian.
\end{itemize}
\end{thm}
\begin{proof}
Assume that $\Delta(G)$ with odd degree vertices.

(i) If $\Delta(G)$ is not a block, by   Theorem~\ref{3.3}  $\Delta(G)$ is not a graph with odd degree vertices. Hence $\Delta(G)$ is a block.
 
(ii) We claim that both $\mid\rho_1 \cup \rho_2\mid$ and $\mid\rho_3 \cup \rho_4\mid$ are even. Suppose not, let us assume that $\mid\rho_1\cup \rho_2\mid$ is odd. As mentioned in Remark~\ref{rem2.2}, $\rho_1\cup \rho_2$ is a complete graph and no prime in $\rho_1$ is adjacent to any prime in $\rho_3\cup \rho_4.$ Hence the vertices in $\rho_1$ are of even degree and so $\Delta(G)$ is not a graph with odd degree vertices, which is a contradiction. Therefore $\mid\rho_1 \cup \rho_2\mid$ is even. 

Suppose that $\mid\rho_3\cup \rho_4\mid$ is odd. As mentioned in Remark~\ref{rem2.2}, $\rho_3\cup \rho_4$ is a complete graph and no prime in $\rho_4$ is adjacent to any prime in
$\rho_1\cup \rho_2.$ Hence the vertices in $\rho_4$ are of even degree and so $\Delta(G)$is not a graph with odd degree vertices, which is a contradiction. Therefore $|\rho_3 \cup \rho_4|$ is even. 

(iii) Suppose the subgraph induced by $\rho_2\cup \rho_3$ is non Eulerian.  One can refer Remark~\ref{rem2.2} for the following facts. No prime in $\rho_1$ is adjacent to any prime in $\rho_3\cup \rho_4.$ Similarly no prime in $\rho_4$ is adjacent to any prime in $\rho_1\cup \rho_2.$  Since $\rho_1 \cup \rho_2$ and $\rho_3 \cup \rho_4$ are complete graphs with even number of vertices. Hence each vertex in $\rho_2$ and $\rho_3$ is of odd degree in sub graphs induced by $\rho_1 \cup \rho_2$ and $\rho_3 \cup \rho_4$. Therefore  the vertices in $\rho_1$ and $\rho_4$ are of odd degree. But the degree of each vertex in $\rho_2$ is the sum of the degree of that vertex in $\rho_2$ in the subgraph induced by $\rho_2\cup \rho_3$ and degree of that vertex in $\rho_2$ in  the subgraph induced by $\rho_1\cup \rho_2$.  Similarly degree of each vertex in $\rho_3$ is the sum of degree of that vertex in $\rho_3$ in the subgraph induced by $\rho_2\cup \rho_3$ and degree of that vertex in $\rho_3$ in the subgraph induced by $\rho_3\cup \rho_4$. By assumption $\rho_2\cup \rho_3$ is non Eulerian and hence there are vertices of odd degree in $\rho_2$ or $\rho_3$ in the subgraph induced by $\rho_2\cup \rho_3.$ Hence some vertices in $\rho_2$  or  $\rho_3$ are of even degree, which is a contradiction to $\Delta(G)$ is  a graph with odd degree vertices. Therefore the vertex induced subgraph $\rho_2\cup \rho_3$ is Eulerian.

Conversely, assume that conditions (i)-(iii) are true. According to Lewis partition, the subgraphs induced by $\rho_1\cup \rho_2$ and $\rho_3\cup \rho_4$ are complete subgraphs of $\Delta(G),$  no prime in $\rho_1$ is adjacent to any prime in $\rho_3\cup \rho_4$ and  no prime in $\rho_4$ is adjacent to any prime in $\rho_1\cup \rho_2.$ Since both $\mid\rho_1 \cup \rho_2\mid$ and $\mid\rho_3 \cup \rho_4\mid$ are even, vertices in $\rho_1$ and $\rho_4$ are of odd degree. Since $\Delta(G)$ is a block, by Theorem~\ref{2.7} we get that $\mid\rho_2\mid, \mid\rho_3\mid\geq 2.$ Since the subgraph induced by $\rho_2 \cup \rho_3$ is Eulerian, we get that vertices in $\rho_2$ and $\rho_3$ are of odd degree. According to Lewis partition, all subsets $\rho_1, \rho_2, \rho_3$ and $\rho_4$ in $\rho(G)$  are non-empty disjoint subsets. Therefore all the vertices in $\rho(G)$ are odd degree. Hence $\Delta(G)$ is a graph with odd degree vertices. \hfill $\square$ 
\end{proof}

\section{Number of  Character Degree Graphs with odd degree}
In this section,  we obtain  the number of  character degree graphs with odd degree in terms of number of vertices. Actually, we obtain below that the number of character degree graphs with $n$ vertices ($n \geq 6$ is even) which are non regular    by assuming that the diameter of $\Delta(G)$ is two. 
 
\begin{thm} \label{3.4}
The character degree  graphs with n vertices ($ n \geq 6$ is even)   for some solvable group $G$  has atleast $1$ non regular  character degree graphs with odd degree,   Here  $\Delta(G)$  is a block with diameter $2$. 
\end{thm}
\begin{proof} In this proof all the character degree graphs are constructed via direct products only  by definition~\ref{def2.8}.
 \newline\textbf{Case i:}$\mid\rho(G)\mid = 6.$ Bissler et. al ~\cite[p. 503] {BL} classified all character degree graphs with $6$ vertices expect $9$ graphs  listed in \cite[p. 509]{BL}. But none of these $9$ graphs are character degree graphs with odd degree.  On the other hand, among  twelve graphs, there exists only one non regular graph with all the vertices as odd degree and the same is given in Figure $2$.
\begin{figure}[ht]
\centering
\begin{tikzpicture}
\draw[fill=black] (0,1) circle (3pt);
\draw[fill=black] (0,-1) circle (3pt);
 \draw[fill=black] (-1,0) circle (3pt);
\draw[fill=black] (-2,0) circle (3pt);
 \draw[fill=black] (1,0) circle (3pt);
\draw[fill=black] (2,0) circle (3pt);
\draw[thick] (0,1) --(0,-1) --(1,0) --(2,0);
\draw[thick] (0,-1) --(-1,0) --(-2,0) ;
\draw[thick] (0,-1) --(-2,0);
\draw[thick] (0,-1) --(2,0);
\draw[thick] (0,1) --(-2,0);
\draw[thick] (0,1) --(2,0);
\draw[thick] (0,1) --(-1,0);
\draw[thick] (0,1) --(1,0);
\end{tikzpicture}
 \caption{Graph with six vertices}
\end{figure}
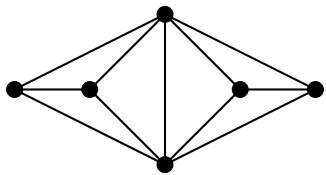
\newline\textbf{Case ii:} $\mid\rho(G)\mid = 8.$ Let $\Delta(G)$ be the character degree graphs as in Figure $2.$  Let $\Delta(H)$ be $K_2$ and $\rho(H)$ are disjoint from $\rho(G)$.  By direct product,  $\Delta(G \times H)$  have $8$ vertices in which $4$ vertices have degree $5$ and remaining $4$ vertices have degree $7$.

The number of  character degree graphs with odd degree which are not regular graph is $1$.
\newline\textbf{Case iii:} $\mid\rho(G)\mid = 10.$ Let $\Delta(G)$ be $8$ vertices graph as in case $3$ and $\Delta(H)$ be $K_2$ and $\rho(H)$ are disjoint from $\rho(G)$.  By direct product,  $\Delta(G \times H)$  have $10$ vertices in which $4$ vertices have degree $7$ and remaining $6$ vertices have degree $9$.

The number of  character degree graphs with odd degree which are not regular graph is $1$.
\newline By induction method,  Let $\Delta(G)$ be the character degree graph with $n-2$ vertices in which $4$ vertices have degree $n-5$ and remaining $n-6$ vertices have degree $n-3$ and $\Delta(H)$ be $K_2$ which are disjoint from $\rho(G)$.  By direct product,  $\Delta(G \times H)$  have $n$ vertices in which $4$ vertices have degree $n-3$ and remaining $n-4$ vertices have degree $n-1$. \hfill $\square$ 

\begin{rem}\normalfont\label{rem2.4} By  Theorem~\ref{2.6}  solvable group in Theorem~\ref{3.4}  has Fitting height  is at least 3.
\end{rem}
\end{proof}

\section*{Acknowledgment}
The second author is partially supported by DST-FIST (Letter No: SR / FST / MSI - $115$ / $2016$ dated: $10$ - $10$ - $2017$ ).

\end{document}